\documentclass[11pt]{amsart}
\usepackage{mathrsfs}
\usepackage{amssymb}
\usepackage[utf8]{inputenc}
\usepackage[T1]{fontenc}
\usepackage{amscd}
\usepackage{amsmath, amssymb}
\usepackage{amsfonts}
\usepackage[colorlinks,linkcolor=blue,citecolor=blue, pdfstartview=FitH]{hyperref}

\usepackage{amscd}
\usepackage{easybmat}
\usepackage{mathrsfs}
\usepackage{amsfonts}
\usepackage{color}
\usepackage{pifont}
\usepackage{upgreek}
\usepackage{bm}
\usepackage{hyperref}
\usepackage{shorttoc}
\usepackage{amsmath,amstext,amsthm,a4,amssymb,amscd}
\usepackage[mathscr]{eucal}
\usepackage{mathrsfs}
\usepackage{epsf}

\numberwithin{equation}{section}

\def\p{\partial}

\def\b{\bar}
\def\mb{\mathbb}
\def\mc{\mathcal}
\def\n{\nabla}

\theoremstyle{plain}
\newtheorem{thm}{Theorem}[section]
\newtheorem{lemma}[thm]{Lemma}
\newtheorem{prop}[thm]{Proposition}

\theoremstyle{definition}

\theoremstyle{definition}

\newcommand{\comment}[1]{}

\usepackage{fancyhdr}
\begin{document}

\title{
  Norm estimates and asymptotic faithfulness of the
  quantum $SU(n)$ representations of the mapping class group
}
\makeatletter
\makeatother
\author{Xueyuan Wan}
\author{Genkai Zhang}

\address{Xueyuan Wan: Mathematical Sciences, Chalmers University of Technology, University of Gothenburg, 41296 Gothenburg, Sweden}
\email{xwan@chalmers.se}

\address{Genkai Zhang: Mathematical Sciences, Chalmers University of Technology, University of Gothenburg, 41296 Gothenburg, Sweden}
\email{genkai@chalmers.se}

\begin{abstract}
We  give a direct proof
for the asymptotic faithfulness of the quantum $SU(n)$ representations of the mapping class group  using peak sections in Kodaira embedding.
We give also estimates on the norm of the parallell transport
of the projective connection on the Verlinde bundle.
The faithfulness has been proved earlier
by J. E. Andersen using
Toeplitz operators of compact K\"ahler manifolds
and by J. March\'e and M. Narimannejad
using skein theory. 
 \end{abstract}
 \thanks{The research by Genkai Zhang was partially supported by
   the Swedish Research Council (VR)}
 \maketitle

\section{Introduction}

Let $\Sigma$ be a closed oriented surface of genus $g\geq 2$ and $p\in\Sigma$. 
We consider the moduli space $M$
of flat $SU(n)$-connections $P$ 
on ${\Sigma\setminus\{p\}}$ with fixed
holonomy a center element $d\in \mb{Z}/n\mb{Z}\cong Z_{SU(n)}$ 
of  $SU(n)$. We assume that $n$ and $d$ are coprime, in the case of $g=2$ we also allow $(n,d)=(2,0)$, namely the $SU(2)$-connections with trivial
holonomy.

There is a canonical symplectic form $\omega$ on $M$ obtained
by integrating wedge product of  Lie algebra $\mathfrak{su}(n)$-valued
connection forms. The natural action of the mapping class group $\Gamma$ 
of $\Sigma$ on $(M, \omega)$ is symplectic. Let $\mc{L}$ be the Hermitian line bundle over $M$ and $\n$ the compatible connection in $\mc{L}$ constructed by Freed \cite{Freed}. By \cite[Proposition 5.27]{Freed}, the curvature of $\n$ is $\frac{\sqrt{-1}}{2\pi}\omega$. Given
 any element $\sigma$ in the Teichm\"uller space $\mc{T}$
the symplectic manifold $M$
can be equipped with a  K\"ahler structure 
so that $\mc{L}$
becomes a holomorphic ample line bundle $\mc{L}_\sigma$. The {\it Verlinde bundle} $\mc{V}_{k}$ is defined by $$\mc{V}_{k}=H^0(M_{\sigma},\mc{L}^k_{\sigma}).$$
It is known by the works of Axelrod-Della Pietra-Witten 
\cite{ADW} and Hitchin \cite{Hitchin} that the projective bundle $\mb{P}(\mc{V}_k)$ is equipped with a natural flat connection.  Since there is an action of the mapping class group $\Gamma$ of $\Sigma$ on $\mc{V}_k$ covering its action on $\mc{T}$, which preserves the flat connection in $\mb{P}(\mc{V}_k)$, we get for each $k$, a finite dimensional
projective representation of $\Gamma$. This sequence of projective representations $\pi^{n,d}_k$, $k\in\mb{N}_+$, is the {\it quantum $SU(n)$ representation} of the mapping class group $\Gamma$. 

V. Turaev \cite{Turaev} conjectured that there should be
no nontrivial element $\phi$
of the mapping class group in the kernel
of all $\pi^{n,d}_k$ for all $k$, keeping $(n,d)$ fixed.
This property is called {\it asymptotic faithfulness} of the quantum $SU(n)$ representations $\pi^{n,d}_k$. In \cite[Theorem 1]{Andersen}, J. E. Andersen proved Turaev's conjecture, namely
the following 

\begin{thm}[{\cite[Theorem 1]{Andersen}}]\label{main theorem}
Let $\pi^{n,d}_k$ be the projective representation of the mapping class group.
	Assume that $n$ and $d$ are coprime or that $(n,d)=(2,0)$ when $g=2$, then 
	\begin{align}
	\bigcap_{k=1}^{\infty}\text{Ker} (\pi^{n,d}_k)=\begin{cases}
&\{1,H\}, \quad g=2,\quad (n, d)=(2, 0)\\
&\{1\},\quad \text{otherwise},
\end{cases}
	\end{align}
where $H$ is the hyperelliptic involution on genus $g=2$ surfaces.
\end{thm}

This theorem is proved
in \cite{Andersen} by considering
the action of the mapping class group on functions on $M$
as symbols of  Toeplitz operators on holomorphic sections
of $\mathcal L_\sigma^k$ for large $k$. A different
proof using skein theorem is given in \cite{MN};
see also
\cite{A-2, A-3, AG} and references therein for further developments.
The existing proofs seem rather involved.
We shall give a somewhat more direct and elementary proof
using peak sections in the Kodaira embedding.

We describe briefly our approach.
Write $\pi^{n,d}_k$ as $\pi_k$ throughout the rest of the paper.
The action of an element $\phi\in \Gamma$ on $\sigma\in \mathcal T$
and $p\in M$ will be all denoted by the same, $\phi(\sigma)$
and $\phi(p)$.

Let $\phi\in \Gamma, \sigma\in \mc{T}$, and $\sigma(t):[0,1]\to\mc{T}$ be a smooth curve connecting $\phi(\sigma)$ and $\sigma$.  Denote by $P_{\phi(\sigma),\sigma(t)}$ the parallel transport  from $\phi(\sigma)$ to $\sigma(t)$ with respect to the projective flat connection (\ref{0.1}) below. For any  $s\in H^0(M_{\sigma},\mc{L}^k_{\sigma})$, set
\begin{align*}
s(t):=P_{\phi(\sigma),\sigma(t)}\circ\phi^*(s)\in H^0(M_{\sigma(t)},\mc{L}^k_{\sigma(t)}).	
\end{align*}
Here $\phi^*$ is the induced action of $\phi$
on the total space of the Verlinde bundle.
For any positive smooth function $\rho: M\to (0,1]$  define a
Hermitian structure on the trivial
bundle $\mc{H}_k=\mc{T}\times C^{\infty}(M,\mc{L}^k)$ by 
\begin{align}
\langle s_1, s_2\rangle_{\rho}=\int_M \rho\cdot(s_1,s_2)\frac{\omega^m}{m!},
\quad
\|s\|^2_{\rho}=\langle s, s\rangle_{\rho}.
\end{align}
We shall study  the variation of $\|s(t)\|^2_{\rho}$ and obtain
\begin{align}\label{Intro1}
	e^{-\frac{C_{\rho}+kC}{k+n}}\|s\|^2_{\rho\circ\phi^{-1}}\leq \|P_{\phi(\sigma),\sigma}\phi^*(s)\|^2_{\rho}\leq e^{\frac{C_{\rho}+kC}{k+n}}\|s\|^2_{\rho\circ\phi^{-1}};
\end{align}
see Proposition \ref{lemma2} below. Here $C_{\rho}$ and $C$ are positive constants independent of $k$. 
We prove that
if $\phi\in \bigcap_{k=1}^{\infty}\text{Ker} \,\pi_k$ then the induced action of $\phi$ on $M$ is the identity. 
First of all it follows that
the representation
 $\phi\to P_{\phi(\sigma),\sigma}\circ\phi^*$ 
is projectively trivial
 on the space $H^0(M_{\sigma},\mc{L}^k_{\sigma})$,
\begin{align*}
P_{\phi(\sigma),\sigma}\circ\phi^*=\pi_k(\phi)=c_k\text{Id}	
\end{align*}
 for some constant $c_k=c_k(\phi)\neq 0$. By taking $\rho=1$ and using (\ref{Intro1}), we get a lower bound of $c_k^2$, i.e. 
$c_k^2\geq e^{-\frac{C_{1}+kC}{k+n}}$, which converges to $e^{-C}$ as $k\to\infty$, so $c_k^2>c$ for some constant $c>0$. 
If $\phi$ on $M$ is not the identity, say $\phi(p)\ne p$
we can construct appropriate 
weight function $\rho$  
and peak section   $s$ at $p$
so that
the right hand side  $e^{\frac{C_{\rho}+kC}{k+n}}\|s\|^2_{\rho\circ\phi^{-1}}$
is arbitrarily smaller than $e^{-C}$ while as
 $\|P_{\phi(\sigma),\sigma}\circ\phi^*(s)\|^2_{\rho}=c_k^2\|s\|^2_{\rho}$ 
has a uniform lower bound
$e^{-C}$, a contradiction
to (\ref{Intro1}). Thus $\phi$ acts as identity on $M$, and it follows further  by standard arguments
that $\phi$ itself is the identity element in $\Gamma$
 under the assumption on $\{g, n, d\}$ or a hyperelliptic involution
for genus $g=2$ surfaces.

We note that even though our proof is simpler
than Andersen's proof  \cite{Andersen}
but the underlying
ideas are very much related;  indeed
Andersen used  the result of
Bordemann-Meinrenken-Schlichenmaier
\cite{BMS}
on  norm estimates
of Teoplitz operators $T_f$ which are
based on coherent states, namely  specific kinds
of peak sections. Finally we mention
that constructing representations
of the mapping class group and the study
of faithfulness of the corresponding representations
are of much
interests; see \cite{FWW, MN}
and references therein.

This article is organized as follows. In Section \ref{sec1}
 we fix notation and recall some basic facts on the Verlinde bundle, the projective flat connection and  peak sections. 
Theorem \ref{main theorem}. 
is proved in  Section \ref{sec2}.

We would like to thank Jorgen Ellegaard Andersen for some informative
explanation of his results.

\section{Preliminaries}\label{sec1}

The results in this section 
can be found in  \cite{Andersen, ADW, Hitchin, Ma2006, Tian}
and references therein.

Let $\Sigma$ be a closed oriented surface of genus $g\geq 2$ and
$p_0\in\Sigma$. 
 Let $d\in \mb{Z}/n\mb{Z}\cong Z_{SU(n)}=\{cI, c^n=1\}$,
 the center of $SU(n)$. We assume that $n$ and $d$ are coprime,
 in the case of $g=2$ we also allow $(n,d)=(2,0)$. Let $M$ be the moduli space of flat $SU(n)$-connections $P$ on  ${\Sigma\setminus\{p_0\}}$ with 
fixed holonomy $d$ around $p_0$. $M$ is then a  compact smooth manifold of dimension $m=(n^2-1)(g-1)$
with tangent vectors given by the Lie algebra
$\mathfrak{su}(n)$-valued connection $1$-forms.

There is a canonical symplectic form $\omega$ on $M$
by taking the trace of the integration of
products of $1$-forms, the natural action of the mapping class group $\Gamma$ on $M$ is symplectic. Let $\mc{L}$ be the Hermitian line bundle over $M$ and 
$\n$ the compatible connection in $\mc{L}$ 
with curvature 
  $\frac{\sqrt{-1}}{2\pi}\omega$; see \cite[Proposition 5.27]{Freed}.
The induced connection in $\mc{L}^k$ will also be denoted
by $\n$.

Let $\mc{T}$ be the Teichm\"uller space of $\Sigma$
parametrizing all marked complex structures on $\Sigma$. 
By a classical result of Narasimhan and Seshadri \cite{NS}
each $\sigma\in \mc{T}$ induces a K\"ahler structure on $M$ and
 thus a K\"ahler manifold $M_{\sigma}$. 
By using the $(0,1)$-part of $\n$, the bundle
$\mc{L}$ is then equipped with a holomorphic structure, which we denote by $\mc{L}_{\sigma}$.
Thus the manifold $\mc{T}$ 
also parameterizes  K\"ahler structures $I_{\sigma}$, $\sigma\in\mc{T}$ on $(M,\omega)$ and the holomorphic line bundles $\mc{L}_{\sigma}$. For any positive integer $k$
the {\it Verlinde bundle} $\mc{V}_k$
is a finite dimensional subbundle
of the trivial bundle $\mc{H}_k=\mc{T}\times C^{\infty}(M,\mc{L}^k)$
given by
$$\mc{V}_k(\sigma)=H^0(M_{\sigma},\mc{L}^k_{\sigma}),\,
\sigma\in\mc{T}.
$$ 
By the results
of Axelrod, Della Pietra, Witten \cite{ADW} and Hitchin \cite{Hitchin}, there is a projective flat connection in $\mc{V}_k$ given by 
\begin{align}\label{0.1}
	\hat{\n}_v=\hat{\n}^t_v-u(v), \quad v\in T(\mc{T}),
\end{align}
 where $\hat{\n}^t$ is the trivial connection in $\mc{H}_k$.
The second term  $u(v)$  is given by \cite[Formula (7)]{Andersen}, 
\begin{align}\label{0.2}
	u(v)=\frac{1}{2(k+n)}\left(\sum_{r=1}^R\n_{X_r(v)}\n_{Y_r(v)}+\n_{Z(v)}+nv[F]\right)-\frac{1}{2}v[F],
\end{align}
where $F:\mc{T}\to C^{\infty}(M)$  is a smooth function
 such that $F(\sigma)$ is real-valued on $M$ for all $\sigma\in\mc{T}$,
 $\{X_r(v), Y_r(v), Z(v)\}\subset 
C^{\infty}(M_{\sigma}, T)$ are a finite set of
vector fields of $M_\sigma$ taking value in the holomorphic tangent space $T$.
 
Since $\mc{L}_{\sigma}$ is an ample line bundle over $M_{\sigma}$, one may take a large $k$ such that $\mc{L}^k_{\sigma}$ is a very ample line bundle. Then the Kodaira embedding is given by 
\begin{align*}
\Phi^k_{\sigma}: M\to \mb{P}(H^0(M_{\sigma},\mc{L}^k_{\sigma})^*),\quad p\mapsto\Phi^k_{\sigma}(p)=\{s\in H^0(M_{\sigma},\mc{L}^k_{\sigma}), s(p)=0\}.
\end{align*}

A {\it peak section} $s_p^k\in H^0(M_{\sigma},\mc{L}^k_{\sigma})$  of $\mc{L}_{\sigma}^k$ at a 
point $p\in M$ is a unit norm generator of the orthogonal complement
 of $\Phi^k_{\sigma}(p)$ 
such that $$|s^k_p(p)|^2=\sum_{i=1}^{N_k}|s_i(p)|^2,$$
where $N_k=\text{dim}H^0(M_{\sigma},\mc{L}^k_{\sigma})$ and 
$\{s_i\}_{1\leq i\leq N_k}$ is an orthonormal basis of $H^0(M_{\sigma},\mc{L}^k_{\sigma})$ with respect to the standard $L^2$-metric;
see  \cite[Definition 5.1.7]{Ma2006}. 
The existence of peak sections is well-known, and for any sequence $\{r_k\}$ with $r_k\to 0$ and $r_k\sqrt{k}\to\infty$, one has
\begin{align}\label{0.3}
\int_{B(p,r_k)}|s_{p}^k(x)|^2\frac{\omega^n}{n!}=1-o(1),\quad \text{for}\quad k\to\infty;	
\end{align}
see e. g. \cite[Formula (5.1.25)]{Ma2006} and \cite[Lemma 1.2]{Tian}.

\section{A direct approach to the asymptotic faithfulness}\label{sec2}

In this section, we will present an elementary proof of  Theorem \ref{main theorem} using peak sections.

Fix $\sigma\in \mathcal T$. For any $\phi\in \Gamma$, 
the mapping class group of $\Sigma$, 
let $\sigma(t):[0,1]\to\mc{T}$ be a smooth curve with $\sigma(0)=\phi(\sigma)$, $\sigma(1)=\sigma$. For any  $s\in H^0(M_{\sigma},\mc{L}^k_{\sigma})$, set
\begin{align}\label{0.5}
s(t):=P_{\phi(\sigma),\sigma(t)}\circ\phi^*(s)\in H^0(M_{\sigma(t)},\mc{L}^k_{\sigma(t)}),	
\end{align}
where $P_{\phi(\sigma),\sigma(t)}$ is the parallel transport  from $\phi(\sigma)$ to $\sigma(t)$ with respect to the projective flat connection (\ref{0.1}).
For any positve smooth function $\rho: M\to (0,1]$ we  define a
rescaled Hermitian structure on $\mc{H}_k$ by 
\begin{align}\label{2.0}
\langle s_1, s_2\rangle_{\rho}=\int_M \rho\cdot(s_1,s_2)\frac{\omega^m}{m!},
\end{align}
and denote $\|s\|^2_{\rho}=\langle s, s\rangle_{\rho}$, where $(\cdot,\cdot)$ denotes the pointwise inner product of the Hermitian  line bundle $\mc{L}^k$.
(We note that the question of projectiveness
of the norm (\ref{2.0}) 
 with respect 
to the connection 
(\ref{0.1}) 
is systematically studied in \cite{R}.)

For a $(1, 0)$-vector field $X$ on $M_\sigma$
let $X^*$ denote the dual $1$-form of $X$
such that $ X^*(X)
=|X|^2_{\omega}$. Denote $\Lambda_t$
 the adjoint of multiplication operator
        $\omega\wedge \bullet$
by the K\"ahler metric $\omega$.

\begin{lemma}\label{lemma1} We have the following
estimate for the differential operator $u$ along $\sigma(t)$, 
$$\left|\langle u(\sigma'(t))s(t),s(t)\rangle_{\rho}\right|\leq \frac{C_{\rho}+kC}{2(k+n)}\|s(t)\|^2_{\rho},$$
where the constants $C=\max_{[0,1]\times M}\left|\frac{\p F(\sigma(t))}{\p t}\right|$ and 
$$C_{\rho}=\max_{[0,1]\times M}|\Lambda_t\b{\p}_t(Z(\sigma'(t))^*\rho)\rho^{-1}|+ \sum_{r=1}^R\max_{[0,1]\times M}\left|\Lambda_t\b{\p}_t\left(Y_r(\sigma'(t))^*\Lambda_t\b{\p}_t(X_r(\sigma'(t))^*\rho)\right)\rho^{-1}\right|$$
are independent of $k$.
\end{lemma}
\begin{proof}
By (\ref{0.2}) and (\ref{2.0}) we have
	\begin{align}\label{0.7}
\begin{split}
	&\quad \left|\langle u(\sigma'(t))s(t),s(t)\rangle_{\rho}\right|=\left|\int_M( \rho u(\sigma'(t))s(t),s(t))\frac{\omega^m}{m!}\right|\\
	&\leq \frac{1}{2(k+n)}\left|\int_M\left(\n_{Z(\sigma'(t))}s(t),\rho s(t)\right)\frac{\omega^m}{m!}\right|\\
	&\quad +\frac{1}{2(k+n)}\left|\int_M\left(\sum_{r=1}^R\n_{X_r(\sigma'(t))}\n_{Y_r(\sigma'(t))}s(t),\rho s(t)\right)\frac{\omega^m}{m!}\right|\\
	&\quad +\frac{k}{2(k+n)}\left|\int_M\rho\cdot\left(\frac{\p F(\sigma(t))}{\p t}s(t),s(t)\right)\frac{\omega^m}{m!}\right|.
	\end{split}
\end{align}
The tangent vectors $X, Y, Z$ are $(1, 0)$-vectors
and $\nabla$ above can all be replaced by $\nabla^{(1, 0)}$,
which we still denote by $\nabla$. By \cite[Chapter VII, Theorem (1.1)]{Dem},
        the adjoint of $\n$  on forms
         is $\n^{(1,0),*}=\sqrt{-1}[\Lambda_t, \b{\p}_t]$,
         and  so it is 
         $\n^{(1,0),*}=\sqrt{-1}\Lambda_t \b{\p}_t$ on $(0, 1)$-forms.

The first term in the RHS of (\ref{0.7}) can be estimated as
\begin{align}\label{1.7}
\begin{split}
	\left|
\int_M\left(\n_{Z(\sigma'(t))}s(t),\rho s(t)\right)\frac{\omega^m}{m!}
\right|
&=
\left|
\langle \n_{Z(\sigma'(t))}s(t), \rho s(t)\rangle
\right|
\\
&=
\left|
\langle s(t), \n^*(Z(\sigma'(t))^*\rho s(t))\rangle
\right|
\\
&=
\left|
\langle s(t), \sqrt{-1}\Lambda_t\b{\p}_t(Z(\sigma'(t))^*\rho)\rho^{-1} \cdot\rho s(t)\rangle
\right|\\
	&\leq \max_{[0,1]\times M}|\Lambda_t\b{\p}_t(Z(\sigma'(t))^*\rho)\rho^{-1}| \cdot\|s(t)\|^2_{\rho},
	\end{split}
\end{align}
where the third equality holds since $s(t)$ is a holomorphic section of  $\mc{L}^k_{\sigma(t)}$, i.e. $\b{\p}_t s(t)=0$.
Similarly the second term is bounded by
\begin{align}\label{1.8}
\begin{split}
	&\quad\left|\int_M\left(\sum_{r=1}^R\n_{X_r(\sigma'(t))}\n_{Y_r(\sigma'(t))}s(t),\rho s(t)\right)\frac{\omega^m}{m!}\right|\\
	&\leq \sum_{r=1}^R\left|\left\langle \n_{X_r(\sigma'(t))}\n_{Y_r(\sigma'(t))}s(t),\rho s(t)\right\rangle\right|\\
	&=\sum_{r=1}^R\left|\left\langle s(t), \n^*Y_r(\sigma'(t))^*\n^*X_r(\sigma'(t))^*\rho s(t)\right\rangle\right|\\
	&=\sum_{r=1}^R\left|\langle s(t),-\Lambda_t\b{\p}_t\left(Y_r(\sigma'(t))^*\Lambda_t\b{\p}_t(X_r(\sigma'(t))^*\rho)\right)\rho^{-1}\cdot \rho s(t)\rangle\right|\\
	&\leq \sum_{r=1}^R\max_{[0,1]\times M}\left|\Lambda_t\b{\p}_t\left(Y_r(\sigma'(t))^*\Lambda_t\b{\p}_t(X_r(\sigma'(t))^*\rho)\right)\rho^{-1}\right|\cdot\|s(t)\|^2_{\rho}.
\end{split}	
\end{align}
For the last term in the RHS of (\ref{0.7}), we have
\begin{align}\label{1.9}
\left|\int_M\rho\cdot\left(\frac{\p F(\sigma(t))}{\p t}s(t),s(t)\right)\frac{\omega^m}{m!}\right|	\leq \max_{[0,1]\times M}\left|\frac{\p F(\sigma(t))}{\p t}\right|\cdot\|s(t)\|^2_{\rho}.
\end{align}

Substituting (\ref{1.7}), (\ref{1.8}) and (\ref{1.9}) into (\ref{0.7}), we obtain
\begin{align*}
\begin{split}
	&\quad \left|\langle u(\sigma'(t))s(t),s(t)\rangle_{\rho}\right|\\
	&\leq \frac{1}{2(k+n)}\max_{[0,1]\times M}|\Lambda_t\b{\p}_t(Z(\sigma'(t))^*\rho)\rho^{-1}|\cdot \|s(t)\|^2_{\rho}\\
	&\quad + \frac{1}{2(k+n)}\sum_{r=1}^R\max_{[0,1]\times M}\left|\Lambda_t\b{\p}_t\left(Y_r(\sigma'(t))^*\Lambda_t\b{\p}_t(X_r(\sigma'(t))^*\rho)\right)\rho^{-1}\right|\cdot \|s(t)\|^2_{\rho}\\
	&\quad+\frac{k}{2(k+n)}\max_{[0,1]\times M}\left|\frac{\p F(\sigma(t))}{\p t}\right|\cdot\|s(t)\|^2_{\rho}\\
	&=\frac{C_{\rho}+kC}{2(k+n)}\|s(t)\|^2_{\rho},
	\end{split}
\end{align*}
completing the proof.
\end{proof}

\begin{prop}\label{lemma2}
We have the following estimate for the norm
of the parallel transport 
$P_{\phi(\sigma),\sigma}$,
	\begin{align}\label{1.1}
	e^{-\frac{C_{\rho}+kC}{k+n}}\|s\|^2_{\rho\circ\phi^{-1}}\leq 
\|P_{\phi(\sigma),\sigma}
\phi^*(s)\|^2_{\rho}\leq e^{\frac{C_{\rho}+kC}{k+n}}\|s\|^2_{\rho\circ\phi^{-1}},
\end{align}
for all $s\in H^0(M_\sigma, \mathcal L_\sigma^k)$.
\end{prop}
\begin{proof}
Using    the definition of $s(t)$ in (\ref{0.5}) we have
\begin{align}\label{0.4}
\hat{\n}_{\sigma'(t)}s(t)=0.	
\end{align}
By (\ref{0.1}) and (\ref{0.4})  we deduce that  
\begin{align*}
\begin{split}
	\frac{d}{dt}\|s(t)\|_{\rho}^2
	&=\langle \hat{\n}^t_{\sigma'(t)}s(t),s(t)\rangle_{\rho}+\langle s(t),\hat{\n}^t_{\sigma'(t)}s(t)\rangle_{\rho}\\
	&=\int_M(\rho u(\sigma'(t))s(t),s(t))+( s(t),\rho u(\sigma'(t))s(t))\frac{\omega^m}{m!}\\
	&=2\text{Re}\langle u(\sigma'(t))s(t),s(t)\rangle_{\rho}.
\end{split}	
\end{align*}
This is treated in  Lemma \ref{lemma1} and we find
\begin{align*}
-\frac{C_{\rho}+kC}{k+n}\|s(t)\|^2_{\rho}\leq \frac{d}{dt}\|s(t)\|^2_{\rho}\leq \frac{C_{\rho}+kC}{k+n}\|s(t)\|^2_{\rho}.
\end{align*}
Hence
\begin{align}\label{0.11}
	e^{-\frac{C_{\rho}+kC}{k+n}}\|s(0)\|^2_{\rho}\leq \|s(1)\|^2_{\rho}\leq e^{\frac{C_{\rho}+kC}{k+n}}\|s(0)\|^2_{\rho}.
\end{align}

Now $\sigma(t)$ is a curve from $\phi(\sigma)$ to $\sigma$, $P_{\phi(\sigma),\sigma(0)}=
P_{\phi(\sigma),\phi(\sigma)}=
\text{Id}$, $\sigma(1)=\sigma$,  and
\begin{align}\label{0.9}
s(0)=\phi^*(s),\quad s(1)=P_{\phi(\sigma),\sigma}\phi^*(s).
\end{align}
The norm of $s(0)$ is given by
\begin{align}\label{0.10}
\begin{split}
\|s(0)\|^2_{\rho}&=\|\phi^*s\|^2_{\rho}=\int_M \rho|\phi^*s|^2\frac{\omega^m}{m!}=\int_M \rho|s\circ\phi|^2\frac{\omega^m}{m!}\\
&=\int_M \rho\circ \phi^{-1}|s|^2\frac{\omega^m}{m!}=\|s\|^2_{\rho\circ\phi^{-1}}.	
\end{split}
\end{align}
Here we have used the fact that
 $\phi$ induces a symplectomorphism of $M$, i.e. $\phi^*\omega=\omega$.

 Combining  (\ref{0.9}) and (\ref{0.11})  we find
 the estimate
\begin{align*}
		e^{-\frac{C_{\rho}+kC}{k+n}}\|s\|^2_{\rho\circ\phi^{-1}}\leq \|P_{\phi(\sigma),\sigma}\phi^*(s)\|^2_{\rho}\leq e^{\frac{C_{\rho}+kC}{k+n}}\|s\|^2_{\rho\circ\phi^{-1}}.
\end{align*}
\end{proof}

\vspace{5mm}

We prove now Theorem \ref{main theorem}.

\vspace{3mm}
{\it The proof of Theorem \ref{main theorem}:} 
We consider first the case of $g\geq 3$, $n$ and $d$ are coprime.
Suppose
 $\phi\in \bigcap_{k=1}^{\infty}\text{Ker}\, \pi_k$. We prove that $\phi$
is the identity mapping of $\Gamma$.

The projective representation of the mapping class group  $\Gamma$
is defined via the  flat connection, in particular
 $\Gamma$ acts on the space
of covariant constant sections over Teichm\"uller space, and
\begin{align}\label{0.8}
  P_{\phi(\sigma),\sigma}\circ\phi^*=\pi_k(\phi)=c_k\text{Id},
 \quad c_k\neq 0,
\end{align}
when acting on the element of $H^0(M_{\sigma},\mc{L}^k_{\sigma})$.

By taking $\rho=1$ and using Proposition \ref{lemma2}, we get
\begin{align}\label{1.3}
	e^{-\frac{C_{1}+kC}{k+n}}\leq c_k^2\leq e^{\frac{C_{1}+kC}{k+n}}.
\end{align}
 
We prove first $\phi$ acts on $M$ as identity.
Otherwise suppose $\phi\neq \text{Id}$ as mappings of $M$.
 Then there exists  a point $p\in M$ such that 
  $p\neq \phi^{-1}(p)$.
 Let $V_p, U_p\subset M$ be two small neighborhoods of $p$ with
\begin{align}\label{1.4}
p\in V_p\Subset U_p,\quad \phi^{-1}(V_p)\subset M-U_p.	
\end{align}
Let $\rho: M\to (0,1]$ be a smooth  function on $M$ satisfying 
\begin{align}\label{1.2}
\rho(x)=
\begin{cases}
	&1,\quad x\in V_p,\\
	&\frac{1}{e^{2C}+1},\quad x\in M-U_p.
\end{cases}	
\end{align}

 For each large $k$ we take the initial section $s$ to be the peak section $s^k_p$ of the point $p$. By (\ref{0.8}), (\ref{1.3}), (\ref{1.2}) and (\ref{0.3}), we find
\begin{align}\label{1.5}
\begin{split}
\|P_{\phi(\sigma),\sigma}\circ\phi^*(s^k_p)\|^2_{\rho}&=c_{k}^2\int_M \rho|s_p^k|^2\frac{\omega^m}{m!}\\
&\geq e^{-\frac{C_{1}+kC}{k+n}}\int_{V_p}|s_p^k|^2\frac{\omega^m}{m!}\\
&\geq e^{-\frac{C_{1}+kC}{k+n}}(1-o(1)).
\end{split}
\end{align}
On the other hand, by (\ref{1.4}), (\ref{1.2}) and (\ref{0.3}), we have also
\begin{align}\label{1.6}
\begin{split}
	\|s^k_p\|^2_{\rho\circ{\phi^{-1}}}&=\int_M \rho\circ{\phi^{-1}}|s^k_p|^2\frac{\omega^m}{m!}\\
	&=\int_{V_p} \rho\circ{\phi^{-1}}|s^k_p|^2\frac{\omega^m}{m!}+\int_{M-V_p} \rho\circ{\phi^{-1}}|s^k_p|^2\frac{\omega^m}{m!}\\
	&\leq\frac{1}{e^{2C}+1}\int_{V_p} |s^k_p|^2\frac{\omega^m}{m!}+\int_{M-V_p} |s^k_p|^2\frac{\omega^m}{m!}\\
	&\leq \frac{1}{e^{2C}+1}+o(1).
	\end{split}
\end{align}
Substituting (\ref{1.5}) and (\ref{1.6}) into (\ref{1.1}) we obtain
 \begin{align*}
	e^{-\frac{C_{1}+kC}{k+n}}(1-o(1))\leq e^{\frac{C_{\rho}+kC}{k+n}}\left(\frac{1}{e^{2C}+1}+o(1)\right).
\end{align*}
 
As $k\to \infty$ it gives
\begin{align*}
	e^{-C}\leq e^{C}\cdot\frac{1}{e^{2C}+1}=\frac{e^{-C}}{1+e^{-2C}}<e^{-C},
\end{align*}
which is a contradiction. So
$\phi$ acts on $M$ as the identity. It follows then from
the standard argument \cite{Andersen} that $\phi$ itself is the identity
element in $\Gamma$ (as equivalence class of mappings of $\Sigma$).

Now in the case  $g=2$, $(n, d)=(2, 0)$, the same
proof above concludes that if $\phi\in
\bigcap_{k=1}^{\infty}\text{Ker} (\pi^{2,0}_k)$ then 
it acts trivially on $M$. It
is then either the identity
or the hyper-elliptic involution
$H$; see \cite{Andersen}. On the other hand  $H$ indeed acts
trivially under 
all $\pi^{2,0}_k$ by its definition.
Thus $\bigcap_{k=1}^{\infty}\text{Ker} (\pi^{2,0}_k)=\{1,H\}$.

\rightline{$\Box$}


\begin{thebibliography}{99}


\bibitem{Andersen} J. E. Andersen, {\it Asymptotic faithfulness of the quantum $SU(n)$ representations of the mapping class group}, Annals of Mathematics {\bf 163} (2006), 347-368.

\bibitem{A-2}
\bysame,
{\it 
Toeplitz operators and Hitchin's projectively flat connection.}
 The many facets of geometry, 177-209, Oxford Univ. Press, Oxford, 2010.

\bibitem{A-3}
\bysame
{\it  
Hitchin's connection, Toeplitz operators, and symmetry invariant deformation quantization},
 Quantum Topol. {\bf 3} (2012), no. 3-4, 293-325.

\bibitem{AG}J. E. Andersen and 
N. L. Gammelgaard,
 {\it  
Hitchin's projectively flat connection, Toeplitz operators and the asymptotic expansion of TQFT curve operators}. 
Grassmannians, moduli spaces and vector bundles, 1-24, Clay Math. Proc., 14, Amer. Math. Soc., Providence, RI, 2011.


\bibitem{ADW}
  S. Axelrod, S. Della Pietra, E. Witten,
  {\it
    Geometric quantization of Chern-Simon gauge theory}, J. Differential Geom. {\bf 33} (1991), 787-902.

\bibitem{BMS}
  M. Bordemann, E.  Meinrenken
  and M. Schlichenmaier,
  {\it
    Toeplitz quantization of Kähler manifolds and
    $gl(N), N\to \infty$ limits.}
    Comm. Math. Phys.
    {\bf 165} (1994), no. 2, 281-296.
  
  

\bibitem{Dem} J.-P. Demailly, Complex analytic and differential geometry, available at https://www-fourier.ujf-grenoble.fr/~demailly/manuscripts/agbook.pdf, 2012.

\bibitem{Freed} D. S. Freed, 
{\it 
Classical Chern-Simons theory, Part 1}, Adv. Math. {\bf 113} (1995), 237-303. 

\bibitem{FWW}
  M. H. Freedman, K. Walker and  Z. Wang,
  {\it
    Quantum $SU(2)$ faithfully detects mapping class groups modulo center.
    }
    Geom. Topol. {\bf 6} (2002), 523-539.

\bibitem{Hitchin} N. Hitchin, {\it Flat connections and geometric quantization}, Comm. Math. Phys. {\bf 131} (1990), 347-380. 

\bibitem{Ma2006} X. Ma, G. Marinescu, Holomorphic Morse Inequalities and Bergman Kernels, Birkh\"auser, Basel$\cdot$ Boston$\cdot$ Berlin, 2006.

\bibitem{MN} 
J. March\'e and M. Narimannejad, 
{\it Some asymptotics of topological quantum field theory via skein
              theory},
 {Duke Math. J.}  {\bf 141} ({2008}),
{573-587}.



\bibitem{NS} M. S. Narasimhan, C. S. Seshadri, {\it Holomorphic  vector bundles on a compact Riemann surfaces}, Math. Ann. {\bf 155} (1964), 69-80.


\bibitem{R}    T. R. Ramadas, 
{\it Faltings' construction of the {K}-{Z} connection},
 {Comm. Math. Phys.} {\bf 196}
      ({1998}),
 {133-143}.

\bibitem{Tian} G. Tian, {\it On a set of polarized K\"ahler metrics on algebraic manifolds}, J. Differential Geom. {\bf 32} (1990), 99-130.

\bibitem{Turaev} V. G. Turaev,  Quantum  Invariants of Knots and $3$-manifolds, de Gruyter Studies in Math. {\bf 18}, Walter de Gruyter, Co., Berlin, 1994.

\end{thebibliography}
\end{document}